\documentclass[11pt]{amsart}
 \usepackage{pifont}
\usepackage{latexsym}
\usepackage{amsmath}
\usepackage{amsfonts}
\usepackage{amssymb}
\usepackage[american]{babel}
\usepackage[babel, final]{microtype}
\usepackage{amssymb}
\usepackage[pagewise]{lineno}
\usepackage[pdftex, final]{graphicx}
\usepackage{caption}
\usepackage{subcaption}
\usepackage{pinlabel}
\usepackage{enumerate}
\usepackage{url}

\usepackage[pdftex, dvipsnames]{xcolor}
\usepackage[pdftex,%
  final,%
  colorlinks=true,%
  linkcolor=NavyBlue,%
  citecolor=NavyBlue,%
  filecolor=NavyBlue,%
  menucolor=NavyBlue,%
  urlcolor=NavyBlue,%
  bookmarks=true,%
  bookmarksdepth=3,%
  bookmarksnumbered=true,%
  bookmarksopen=true,%
  bookmarksopenlevel=2,%
]{hyperref}
\input{macros}

\newcommand*{\T}{\mathrm{T}}
\newcommand*{\Li}{\mathrm{L}}
\newcommand*{\tb}{\mathrm{tb}}
\newcommand*{\rot}{\mathrm{rot}}
\newcommand*{\M}{\mathrm{M}}

\newcommand*{\sg}{\mathrm{sg}}

\newcommand*{\B}{\mathrm{B}}
\newcommand*{\R}{\mathbb{R}}
\newcommand*{\xistd}{\xi_{std}}

\newcommand*{\slk}{\mathrm{sl}}

\newcommand*{\D}{\mathrm{D}}

\begin{document}
\title {Transverse links, open books and overtwisted manifolds}
\author[Rima Chatterjee]{Rima Chatterjee}
\address{Mathematisches Institut\\ Universit\"at zu K\" oln\\ Weyertal 86-90, 
  50931 K\"oln, Germany}
\email{\href{mailto:rchatt@math.uni-koeln.de}{rchatt@math.uni-koeln.de}}
\urladdr{\url{http://www.rimachatterjee.com/}}

\begin{abstract}
We prove that transverse links in any contact manifold $(M,\xi)$ can be realized as a sub-binding of a compatible open book decomposition. We define the support genus of a transverse link and prove that the support genus of a transverse knot is zero if there is an overtwisted disk disjoint from it. Next, we find a relationship between the support genus of a transverse link and its Legendrian approximation.
\end{abstract} 
\maketitle
\section{Introduction}
\label{intro}
Knot theory associated to contact 3-manifolds has been a very interesting area of research. We say a knot in a contact 3-manifold is \emph{Legendrian} if it is tangent everywhere to the contact planes and \emph{transverse} if it is everywhere transverse.  Since Eliashberg's classification of overtwisted contact structures \cite{eliashbergovertwist}, the study of overtwisted contact structures and the knots and links in them, has been minimal. However, in recent years overtwisted contact structures have played central roles in many interesting applications such as building achiral Lefschetz fibration \cite{ef}, near symplectic structures on 4-manifolds \cite{gay} and many more. Thus the overtwisted manifolds and the knot theory associated to them has generated significant interest. There are two types of knots/links in an overtwisted contact structure, namely loose and non-loose knots (Also known as non-exceptional and exceptional knots respectively).

A link in an overtwisted contact manifold is loose if its complement is overtwisted and non-loose otherwise. The first explicit example of a non-loose knot is given by Dymara in \cite{dy}. In general, non-loose knots appear to be rare. Although it is known that there exists a non-loose knot in every $3$-manifold (\cite{etnyre-sv},\cite{geiona}) it was not known if every knot type has a non-loose representative until very recently. In \cite{CEMM}, we found conditions when a knot type will have non-loose representatives in any contact 3-manifold. Legendrian approximation and transverse push-off operations are known to be very useful tool in contact geometry to go back and forth between Legendrian and transverse representatives of the same knot type. While it is known that a non-loose transverse knot will always have a non-loose Legendrian approximation, surprisingly, a non-loose Legendrian knot can have a transverse push-off which is loose  \cite{et}.  

In \cite{et}, Etnyre coarsely classified loose, null-homologous Legendrian and transverse knots. In a previous paper, the author extended that result for Legendrian and transverse links \cite{rim1}. The support genus of a Legendrian link is defined to be the minimum genus over all open books compatible with the underlying contact structure such that the Legendrian link can be put on the page and the page framing and the contact framing agree. The author proved that the coarse equivalence class of Legendrian links have support genus zero and constructed examples to show that the converse is not true \cite{rim1}. While support genus of a Legendrian knot was defined in \cite{ona}, nothing was mentioned about transverse knots. This paper investigates transverse knots and links in contact manifolds and associates any transverse link in a contact manifold to a compatible open book depomposition. The following is the main result of the paper:
\begin{theorem}
\label{thm:sgtransverse}
 Suppose $\T$ is any transverse link in $(\M,\xi)$. Then $\T$ is transversely isotopic to the sub-binding of some open book $(\B,\pi)$ supporting $(M, \xi)$.
\end{theorem}
While it is well-known that any  transverse knot can be associated to a compatible open book \cite{bakeretnyre}, it is not so obvious and much more complicated for the links as the proof of the above theorem shows.

 \fullref{thm:sgtransverse} allows us to define the support genus of a transverse link $\T$  following the support genus of a Legendrian link. The support genus of a transverse link in $(\M,\xi)$ is defined to be:\\
$\sg(\T)$=min\{$g\colon$where $g$ is the minimum genus of a page of an open\\ \hspace{5 cm} book  $(\Sigma,\phi)$ supporting $(\M,\xi)$ such that $\T$ can be realized as a sub-binding of $(\Sigma,\phi)$\}.\\

Using the well-known relation of Legendrian and transverse links, we could relate the support genus of a transverse link with the support genus of its Legendrian approximation.
\begin{theorem}
\label{thm:LegTsg}
Suppose $\T$ is a transverse link in $(\M,\xi)$ and $\Li$ is its Legendrian approximation such that its contact framing is less than or equal to the page framing. Then $\sg(\T)=\sg(\Li)$.
\end{theorem}
The above theorem turns out to be quite surprising as one might hope a transverse link to have a support genus greater than its Legendrian approximation.

Onaran proved the following:
\begin{theorem}[\cite{ona}]
Suppose $\Li$ is a loose, null-homologous Legendrian knot. Then $\sg(\Li)=0$.
\end{theorem}

Using \fullref{thm:LegTsg}, as a corollary to the above result, we could prove the following.
\begin{corollary}
\label{cor:trans}
Suppose $\T$ is a loose, null-homologous transverse knot in $(\M,\xi)$. Then $\sg(\T)=0$.
\end{corollary} 

It is well-known that non-loose transverse knots have all of their Legendrian approximations non-loose. But a non-loose Legendrian can have loose transverse push-off. Non-loose Legendrian unknots do not have any non-loose transverse push-off \cite{et}. On the other hand, in \cite{EMM} the authors found examples of  non-loose transverse torus knots in certain overtwisted contact structures which appears as transverse push-offs of non-loose Legendrian in the same contact structure. But they also have examples where the transverse push-offs of the torus knots become loose.
Nothing is known about the contidions that can guarentee non-loose transverse push-off. Non-zero support genus of a Legendrian knot could be an obstruction for a transverse push-off to be loose. We ask the following question:

{\bf Question:} If a non-loose Legendrian knot has $\sg>0$ then is it's transverse push-off always non-loose?

\subsection{Organization of the paper}
The paper has been organized in the following way: In \fullref{sec:basics}, we had a brief discussion of contact geometry and open book decomposition. In \fullref{sec:main}, we proved our main result \fullref{thm:sgtransverse} and end with a proof of \fullref{thm:LegTsg} and \fullref{cor:trans}.
\vspace{-2 mm}

\subsection{Acknowledgement}
I would like to thank Shea Vela-Vick  and John Etnyre for several enlightening conversations. I would also like to thank the anonymous referee for their helpful suggestions and insightful comments. This research is partially supported by NSF Grant 1907654 and the SFB/TRR 191 ``Symplectic Structures in Geometry, Algebra and Dynamics, funded by the Deutsche Forschungsgemeinschaff (Project- ID 281071066-TRR 191)''.

\section{Basics on contact geometry}
\label{sec:basics}
In this section, we briefly mention the preliminaries of contact geometry and open book decompositions. For more details the reader should check \cite{etknot}, \cite{etcontact} and \cite{etnyrelectures}.
\subsection{Contact structures}
A contact structure $\xi$ on an oriented $3$-manifold $\M$ is a nowhere integrable $2$-plane field and we call ($\M,\xi$) a contact manifold. We assume that the plane fields are co-oriented, so $\xi$ can be expressed as the kernel of some global one form $\alpha$. In this case, the non-integrability condition is equivalent to $\alpha\wedge d\alpha > 0$. 
There are two types of contact structures--tight and overtwisted.
An overtwisted disk is a disk embedded in a contact manifold $(\M,\xi$) such that $\xi$ is tangent to the disk along the boundary. We call a contact manifold overtwisted if it contains an overtwisted disk. Otherwise we call it tight.

Though only few results are known about classifying tight contact structures on manifolds, overtwisted contact structures are completely classified by Eliashberg.
\begin{theorem}{(Eliashberg, \cite{eliashbergovertwist})}Two overtwisted contact structures are isotopic if and only if they are homotopic as plane fields. Moreover, every homotopy class of oriented 2-plane fields contains an overtwisted contact structure.
\end{theorem}

\subsection{Legendrian links}
A link $\Li$ smoothly embedded in $(\M,\xi)$ is said to be Legendrian if it is everywhere tangent to $\xi$. 
For the purpose of this paper, by classical invariants of a link we refer to the classical invariants of its components. To define the classical invariants all the components of the link must be null-homologous. 
 The classical invariants of a Legendrian knot are the topological knot type, \emph{Thurston--Benniquin invariant} $\tb(\Li)$ and \emph{rotation number} $\rot(\Li)$. The invariant
$\tb(\Li)$ measures the twisting of the contact framing relative to the framing given by the Seifert surface of $\Li$.  
The other classical invariant $\rot(\Li)$ is defined to be the winding of $\T\Li$ after trivializing $\xi$ along the Seifert surface. 
\begin{figure}[!htbp]
\centering
\includegraphics[scale=0.2]{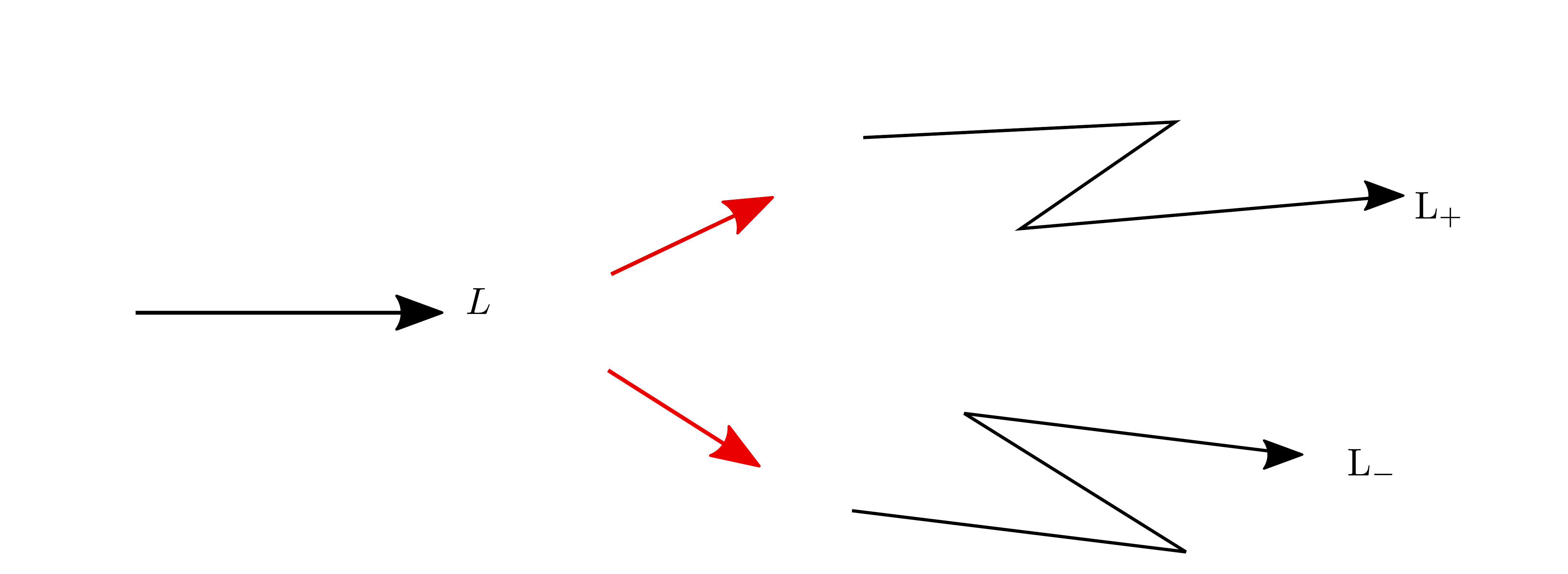}
\caption{Stabilizations of a Legendrian knot.}
\label{fig:stabilization}
\end{figure}
Stabilization of a link can be done by stabilizing any of the link components.
By the standard neighborhood theorem of a Legendrian knot, locally one can identify any Legendrian link component of $\Li$ with the $x$-axis in $\R^3$. Stabilization is a local operation as shown in \fullref{fig:stabilization}. The modification on the top right-side is called the positive stabilization and denoted as $\Li_+$. The modification on the bottom right-side is known as negative stabilizations and denoted as $\Li_-$. It does not matter which order the stabilizations are being done, it just matters where those are being done. The effect of the stabilizations on the classical invariants are as follows:
\[\tb(\Li_\pm))=\tb(\Li)-1 \quad \text{and}\quad \rot(\Li_\pm)=\rot(\Li)\pm 1.\]
\subsection{Transverse link and its relationship with a Legendrian link}
A link $\T$ in ($\M,\xi$) is called transverse (positively) if it intersects the contact planes transversely with each intersection positive. By classical invariant of a transverse link, we will refer to the classical invariants of its components.  There are two classical invariants for transverse knot, the topological knot type and the {\it self-linking number} $\slk(\T)$. Self-linking number is defined for null-homologous knots. Suppose $\Sigma$ is a Seifert surface of a transverse knot. As $\Sigma|_\xi$ is trivial, we can find a non-zero vector field $v$ over $\Sigma$ in $\xi$. Let $\T'$ be a copy of $\T$ obtained by pushing $\T$ slightly in the direction of $v$. The self-linking number $\slk(\T)$ is defined to be the linking number of $\T$ with $\T'$.

Legendrian and transverse links are related by the operations known as transverse push-off and Legendrian approximation. The classical invariants of a Legendrian link component and its transverse push-off are related as follows:
\[ \slk(\T_\pm)=\tb(\Li)\mp\rot(\Li)\] where $\T_\pm$ denotes the positive and negative transverse push-offs.
 In this paper, if we mention a transverse push-off it is always the positive transverse push-off unless explicitly stated otherwise. Note that, while a transverse push-off is well defined, a Legendrian approximation is only well defined up to negative stabilizations. 
\subsection{Open book decomposition and supporting contact structures}


Recall an \emph{open book decomposition} of a $3$-manifold $\M$ is a triple $(\B,\Sigma,\phi)$ where $\B$ is a link in $\M$ such that $\M\setminus\B$ fibers over the circle with fiber $\Sigma$ and monodromy $\phi$ so that $\phi$ is identity near the boundary and each fiber of the fibration is a Seifert surface for $\B$. By saying $\phi$ is the monodromy of the fibration we mean that $\M\setminus\B= \Sigma\times[0,1]/\sim$ where $(x,1)\sim(\phi(x),0)$.
 The fibers of the fibration are called \emph{pages} of the open book and $\B$ is called the \emph{binding}.
 Note that, given a diffeomorphism $\phi$ of the surface $\Sigma$ that is fixed near the boundary one might form a mapping torus and glue in solid tori to build a closed $3$-manifold having an open book decomposion $(\Sigma,\phi)$. So from now on we will drop the term $B$ from the notation and use the pair $(\Sigma,\phi)$ as the open book. 
 
 Given an open book $(\Sigma,\phi)$ for $\M$, let $\Sigma'$ be $\Sigma$ with a $ 1$-handle attached. Suppose $c$ is a simple closed curve that intersects the co-core of the attached 1-handle exactly once. Set $\phi'=\phi\circ D_c^+$, where $D_c^+$ is a right handed Dehn-twist along $c$. The new open book $(\Sigma',\phi')$ is known as the \emph{positive stabilization} of $(\Sigma,\phi)$. If we use $D_c^-$ instead, that will be called a \emph{negative stabilization}. For details check \cite{etplanar}.

We say a contact structure $\xi=\ker\alpha$ on $\M$ is supported by an open book decomposition ($ \Sigma,\phi)$ of $\M$ if
\begin{enumerate}
	
	\item  $d\alpha$ is a positive area form on the page of the open book.
	\item  $\alpha(v)>0$, for each oriented tangent vector to the binding.
\end{enumerate}
Given an open book decomposition of a 3-manifold $\M$, Thurston and Winkelnkemper \cite{thurston} showed how one can produce a compatible contact structure. Giroux proved that two contact structures which are compatible with the same open book are isotopic as contact structures \cite{giroux}. Giroux also proved that two contact structures are isotopic if and only if they are compatible with open books which are related by positive stabilizations.  Thus positive stabilizations of an open book do not change the contact structure it supports, but this is not the case for negative stabilizations. 

It is well known that every closed oriented 3-manifold has an open book decomposition. We can perform an operation called \emph{Murasugi sum} to connect sum two open books and produce a new open book. An interested reader should check \cite{etnyrelectures} for details.

\section{Transverse links and open book decomposition}
\label{sec:main}

We can always associate a Legendrian link in a contact manifold with a compatible open book by putting the Legendrian link on the page of the open book such that the page framing and the contact framing agree. Check \cite{etnyrelectures} for how to do it. But it is not so obvious in the case of transverse links. While Baker and Etnyre showed that we can make any transverse knot a part of a binding of a compatible open book in \cite{bakeretnyre}, it turns out to be much more complicated for a transverse link. We give an explicit proof that any transverse link in any contact manifold can be associated with a compatible open book as mentioned above.

 \begin{theorem}
 \label{thm:binding}
 Let $\T$ be any transverse link in $(\M,\xi)$. Then $\T$ is transversely isotopic to a sub-binding of some open book $(\Sigma,\phi)$ supporting $(M, \xi)$.
 \end{theorem}

 In \cite{bakeretnyre}, the authors prove the following lemma which becomes crucial in the proof of our main theorem.
\begin{lemma}[\cite{bakeretnyre}]
	\label{lemma:crucial}
	
	Let $(\Sigma,\phi)$ be an open book decomposition for a contact manifold $(M,\xi)$ with binding L. Assume $K$ is an oriented Legendrian knot on a page of the open book decomposition such that the page framing and the contact framing agree. Let $\gamma$ be an arc on the page running from one binding component of the open book decomposition to the knot $K$ and approaches the knot $K$ from the right. Set $\alpha$ to be a curve which runs from $L$ along $\gamma$ around $K$ and back to $L$ along  a parallel copy of $\gamma$. The open book $(\Sigma_\alpha,\phi_\alpha)$ obtained from $(\Sigma,\phi)$ by positively stabilizing along $\alpha$ has a binding component $B$ that is the transverse push-off of $K$.
	\end{lemma}

  For notational purpose, we will be using the word ``to the right'' in the following sense: an arc ``to the right'' of a link component will imply that the orientation of the link component followed by the orientation of the arc agrees with the orientation of their intersection point on the page. A set of arc $\Gamma= \{\gamma_1,\gamma_2,\dots \gamma_n\}$ is ``to the right of $\Li$''  means that each of the $\gamma_i$ lies to the right of $\Li_i$ where $\Li_i$ is the $i^{th}$ link component.

\begin{figure}[!htbp]
\centering
\includegraphics[scale=0.07]{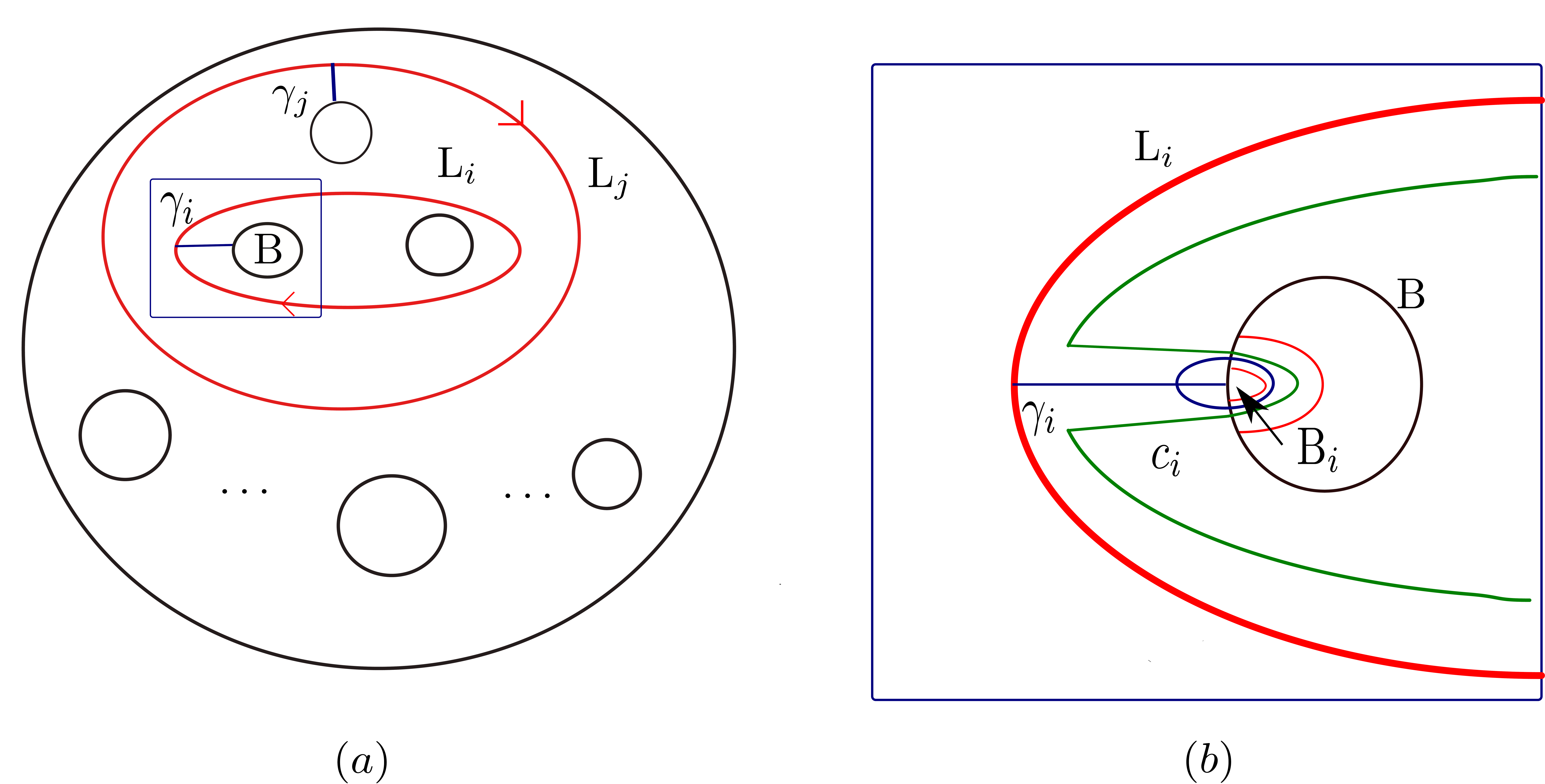}

\caption{(a) A planar open book and the disjoint arcs. (b) Enlarged view of the region inside the box in disjoint arcs lying "to the right".}
\label{fig:planar(a)}
\end{figure}
\begin{proof}[Proof of \fullref{thm:binding}]
Suppose $\T$ is any transverse link in $(\M,\xi)$ and $\Li$ is its Legendrian approximation. Now we can put $\Li$ on the page of an open book supporting $(\M,\xi)$. There are several ways to put a Legendrian link on the page of some supporting open book such that the page framing and contact framing agree (corollary 4.23 in \cite{etnyrelectures}). Observe that all of the components have a fixed orientation.  There are several cases to consider here and we will prove the theorem for each of the cases. Also for notational purpose, we will call two link components `linearly dependent' if they are parallel copies of one another on the page of the open book. Otherwise, they will be `linearly independent'.
\begin{remark}

Note that we might not be able to put the Legendrian link $L$ on the page of every supporting open book of $(\M,\xi)$. For example we can not Legendrian realize an unknot on the open book $(D,id)$ supporting $(S^3,\xistd)$. But we can always find one where we can Legendrian realize it. We start with that particular open book.

\end{remark}

\subsection{\bf Case 1(a)}

\label{ssec:case1a}
[$(\Sigma,\phi)$ is planar and no two link components are linearly dependent] If each of the link components bounds exactly one binding component, then clearly that binding component is its transverse push-off and we are done. If not, we start by finding a set of disjoint arcs $\{\gamma_1,\dots\gamma_n\}$ such that each $\gamma_i$ runs from the closest boundary component of $\Li_i$ to $\Li_i$ and stays ``to the right". Clearly, if all the link components are linearly independent, we can easily find such a set where each of the arcs $\gamma_i$ is disjoint from $\Li_j$ for $j\neq i$ irrespective of the orientation of the link components. See \fullref{fig:planar(a)}(a). Let $c_i$ be the arc which that runs from the binding component $\B_i$ along $\gamma_i$ around $\Li_i$ and back to $\B_i$ following a parallel copy $\gamma_i$ for each $i$. Next we positively stabilize $\B_i$ along $c_i$ (Notice, here $c_i$ we mean $c_i  \cup$ core of the attaching one handle.) and find a link component $\Li_i'$ that runs around the attaching 1-handle exactly once and is topologically isotopic to the link component $\Li_i$ in the whole manifold. Legendrian realize $\Li_i'$. Clearly the new boundary component $\B_i'$ is the transverse push-off of $\Li_i'$. See \fullref{fig:planar(a)}(b). By our choice of $\gamma_i$, $\Li_i'$ is the negative stabilization of $\Li_i$ \cite{bakeretnyre}. Thus they have transversely isotopic transverse push-offs. So $\B_i'$ is $\Li_i$'s transverse push-off as well. The new open book will be $(\Sigma',\phi\circ\D_{c_i})$. After doing the Dehn twist along all such $c_i$'s we find a n-component link $\B'$ such that $\B$ is the transverse push-off of $\Li$. The new monodromy will be given by $\phi\circ\D_{c_1}\circ\dots\D_{c_{n-1}}\circ\D_{c_n}$. By the well-definedness of transverse push-off, $\B'$ is transversely isotopic to $\T$. Observe that, as all of the $c_i$'s are disjoint from each other, the order of Dehn twist does not matter. 
\begin{figure}[!htbp]
\includegraphics[scale=0.08]{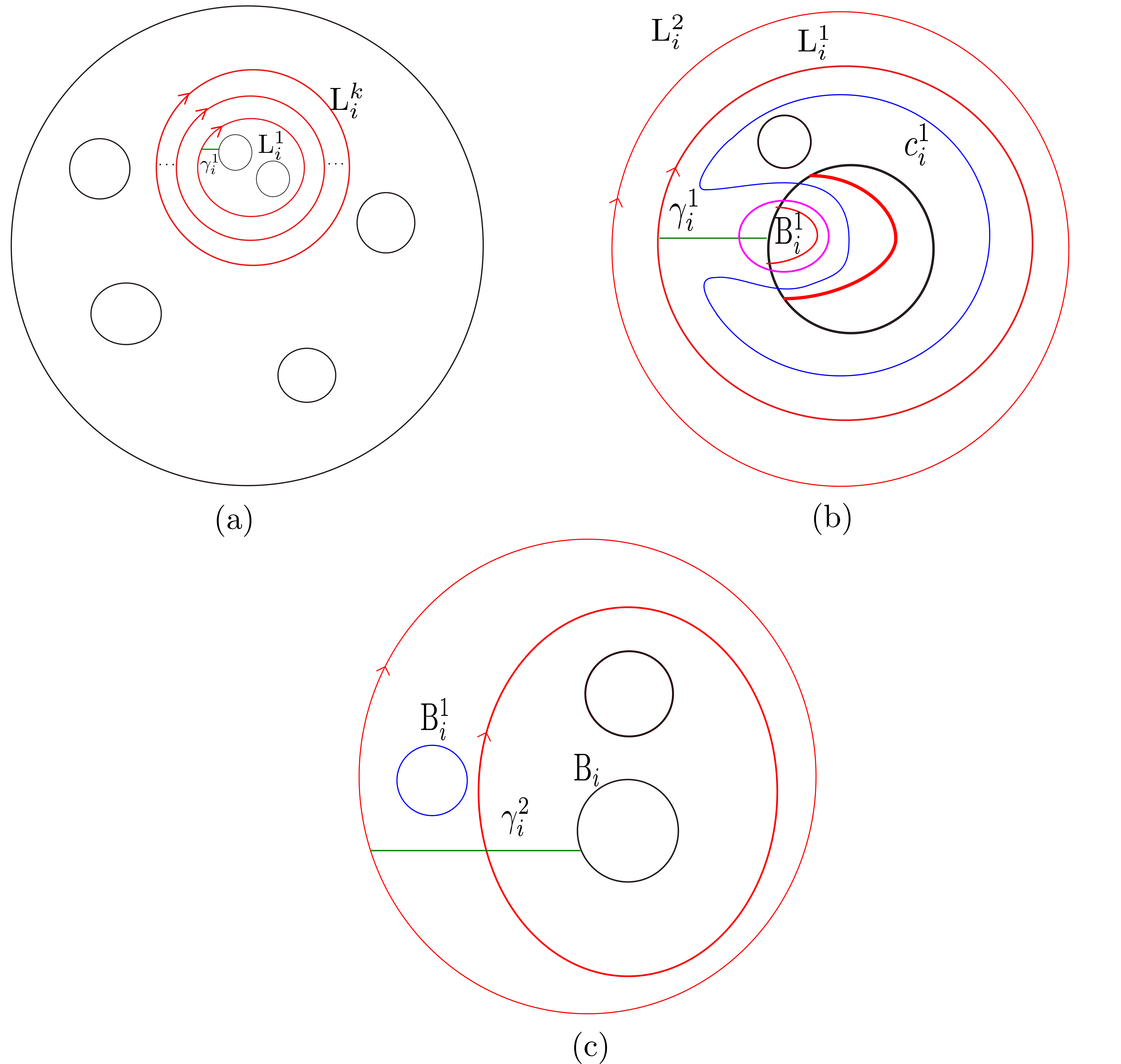}

\caption{(a) A planar open book where $\Li_i$ has $k$ parallel copies. (b) A local picture with two parallel components. We stabilize the open book along one of the innermost boundary component. (c) The new boundary component is $\B_i^1$. Now choose $\gamma_i^2$ on the stabilized open book which only intersects $\Li_i^1$.}
\label{fig:planar2}
\end{figure}

\subsection{\bf Case 1(b)}
\label{ssec:case1b}
[$(\Sigma,\phi)$ is planar and some of the link components are parallel copies of one another] Suppose $\Li_i$ has $k$ parallel copies and we will call them $\Li_i^j$ for $j=1,2,\dots k$. We will treat this case differently. In this case either all the link components are oriented similarly or some of them can have opposite orientations.
\subsubsection{Case 1(b)(i)} 
\label{sssec:case1b(i)}
First we consider the case where all $\Li_i^k$'s are oriented similarly. Check the local picture \fullref{fig:planar2}. Choose the link component that has a boundary component closest to it so that the arc $\gamma_i^1$ lies ``to the right''.  It can be the outer most or the innermost $\Li_i^k$ according to the orientation. Without loss of generality, we assume it to be the innermost component and call it $\Li_i^1$. Let $\B_i$ be its closest boundary component. We call the other components $\Li_i^k$ where $k=2,3\dots $ and order them from innermost to outermost. Now we find $c_i^1$ using $\gamma_i^1$ as before and positively stabilize the open book along $\B_i$ and find $\B_i^1$ as the transverse push-off of  the negatively stabilized $\Li_i^1$ as mentioned in \fullref{ssec:case1a}. Once we find $\B_i^1$, we follow the same procedure on the new open book $( \Sigma',\phi\circ D_{c_i^1})$ and find an arc $\gamma_i^2$ that runs from $\B_i$ to $\Li_i^2$ and stays to the right. Now notice here, this arc could possibly only intersect $\Li_i^1$ and no other link component. As we have already dealt with $\Li_i^1$, this does not create any problem. Now use $\gamma_i^2$ to find $B_i^2$ using the same procedure as before. Inductively, doing so we find a $k$-component link $\B_i^1\sqcup\dots\B_i^k$ such that $\Li_i^1\sqcup\dots\sqcup\Li_i^k$ has transverse push-off transversely isotopic to $\B_i^1\sqcup\dots\B_i^k$. We can do this locally for all the link components which are not linearly independent. Notice, as the curves intersect each other, we strictly need to maintain the order of Dehn twist in this case.
\begin{figure}
\includegraphics[scale=0.14]{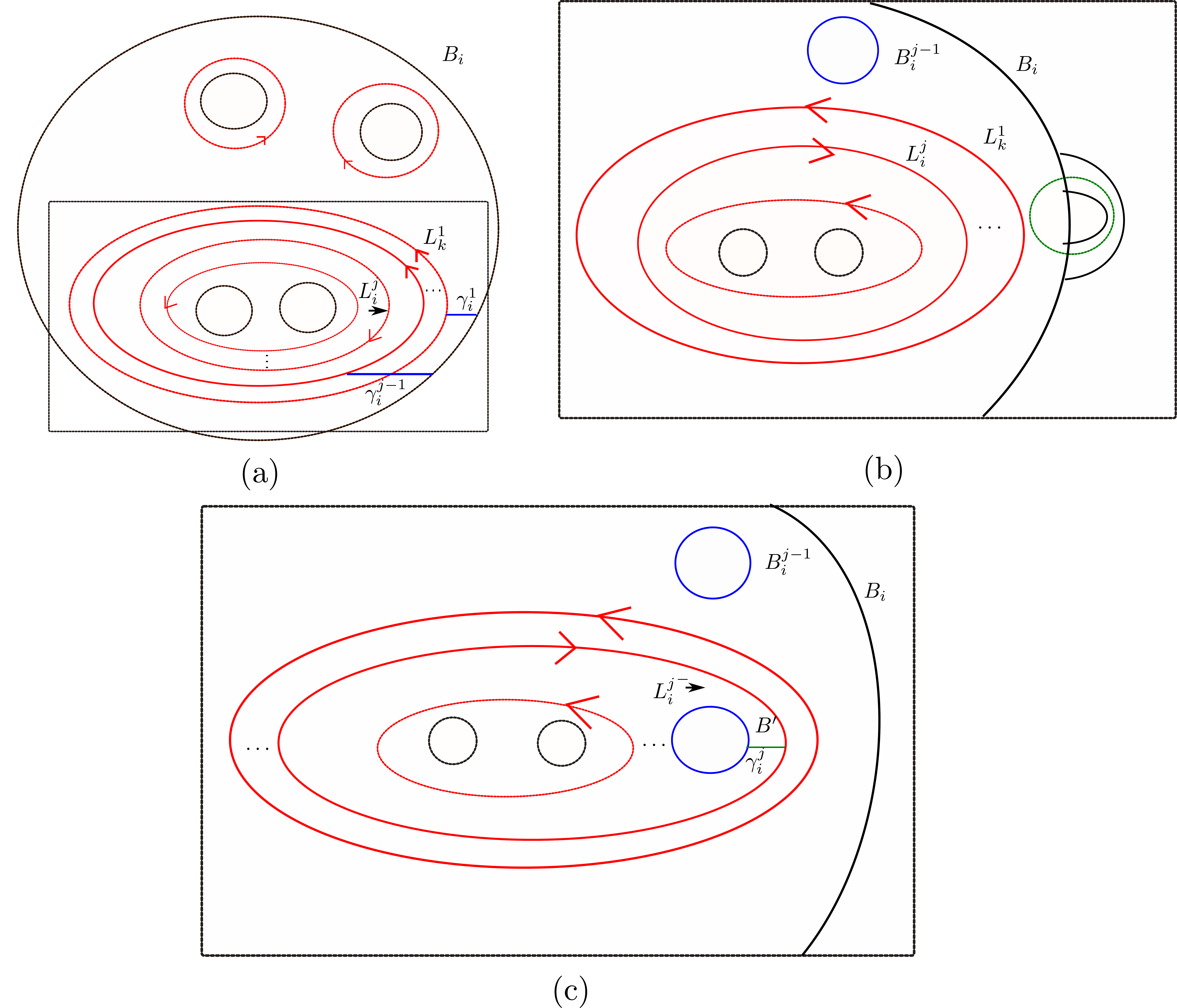}
\caption{(a) A planar open book with linearly dependent link components and distinct orientations. (b) A local picture of the stabilized open book. (c) Push the components over the attaching one handle. $\Li_i^{j-}$ denotes the negatively stabilized component. Now use the new boundary component $\B'$ to find the arc $\gamma_i^j$ to the right.}
\label{fig:distinctorientation}
\end{figure}

\subsubsection{Case 1(b)(ii)}
\label{sssec:differentorientation} In this case, some of the link components are oriented opposite and thus we cannot use the same boundary component for all of them like we did in \fullref{sssec:case1b(i)}. First we choose a link component $\Li_i^k$ such that there exists an arc that joins $\Li_i^k$ to its closest boundary component $\B_i$, lies to the right and is disjoint from all other $\Li_i^j$'s for $j\neq k$. It can be the innermost or outermost link component depending on the orientation. Without loss of generality, suppose it is the outer most component and call it $\Li_i^1$ and the arc $\gamma_i^1$. We name the other parallel components $\Li_i^j$ where $j=2,3,\cdots$ as we move from the innermost component to outermost component. Check \fullref{fig:distinctorientation}. Now we find $ c_i^1$ as before and positively stabilize $\B_i$ along it. Now we have a new boundary component $\B_i^1$ which will be transversely isotopic to the transverse push-off of $\Li_k^1$. Now for the next parallel component $\Li_k^2$, if it has the same orientation, we can easily find an arc $\gamma_i^2$ that lies ``to the right'', intersects $\Li_i^1$ and no other link component. We repeat this process step by step till we find a component which has a different orientation. Suppose $\Li_i^j$ is the link component with different orientation. Notice, now we can not find an arc that lies to the right and is disjoint from $\Li_i^l$ where $l>j$.  
To fix this problem, we first positively stabilize the open book along the boundary component $\B_i$ along a boundary parallel curve and push all $\Li_i^k$'s over the attaching 1-handle where $k=1,2\dots j$. 
This negatively stabilizes $\Li_i^j$. So $\Li_i^j$ and $\Li_i{^j}^-$ will have isotopic transverse push off.  Now using the new boundary component $\B_i'$ coming from the stabilization, we can find an arc $\gamma_i^j$ that lies ``to the right'' of $\Li_i^j$ and is disjoint from all $\Li_i^l$ for $l>j$. Check \fullref{fig:distinctorientation}(c). Now we can continue using $\B_i'$ for all the link components till we find a link component that is oriented differently. Inductively, doing so will give us a transverse link that is a sub-binding of the open book.

\begin{figure}[!htbp]
\includegraphics[scale=.1]{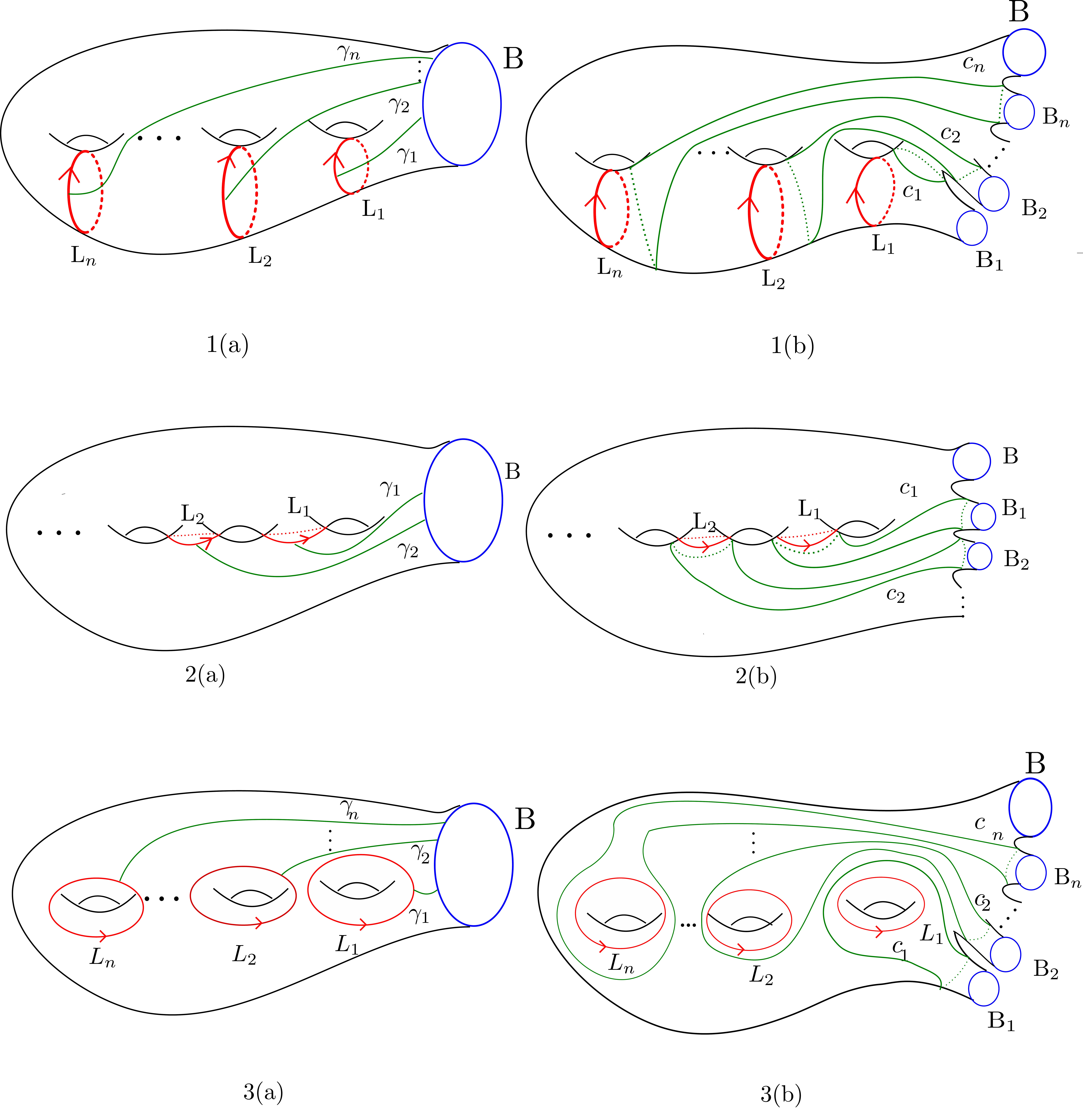}

\caption{All possible cases where the link components are of the same type. On the left, the set of disjoint arcs that lie ``to the right''. On the right the resulting open book with the transverse link $\sqcup_{i=1}^{n}\B_i$.}
\label{fig:genussametype}
\end{figure}
\subsection{\bf Case 2(a)}
\label{ssec:case2a}
[$(\Sigma_g,\phi)$ has $g>0$ and all the link components are linearly independent]

 This is an easy case to deal with. Here we can have all link components of the same type (\fullref{fig:genussametype}) or different type (\fullref{fig:genusmixedtype}). But irrespective of the types and orientation clearly there exists a set of disjoint arcs $\Gamma=\{\gamma_1,\dots\gamma_n\}$ that lies on the ``right'' of $\Li$ and runs from $\Li_i$ to $\B$ as shown in \fullref{fig:genussametype}(a) and \fullref{fig:genusmixedtype}(a). Now we will use $\Gamma$ to find a set of disjoint closed curves $\{ c_1,c_2,\dots c_n\}$ like before and stabilize $\B$ along $c_i$'s for every $i=1,2,\dots n$. Finally we will find an $n$ component link $\B'=\sqcup_{i=1}^{n}\B_i$ as before which is the transverse push-off of $\Li$ for the same reason and thus $\B'$ is transversely isotopic to $\T$. The new open book will have monodromy $\phi\circ\D_{c_1}\circ\cdots\D_{c_{n-1}}\circ\D_{c_n}$ and as all the curves $\{c_1,c_2,\dots, c_n \}$ are disjoint from the link components the order of Dehn twist does not matter in this case.
\begin{figure}[!htbp]
\centering
\includegraphics[scale=0.1]{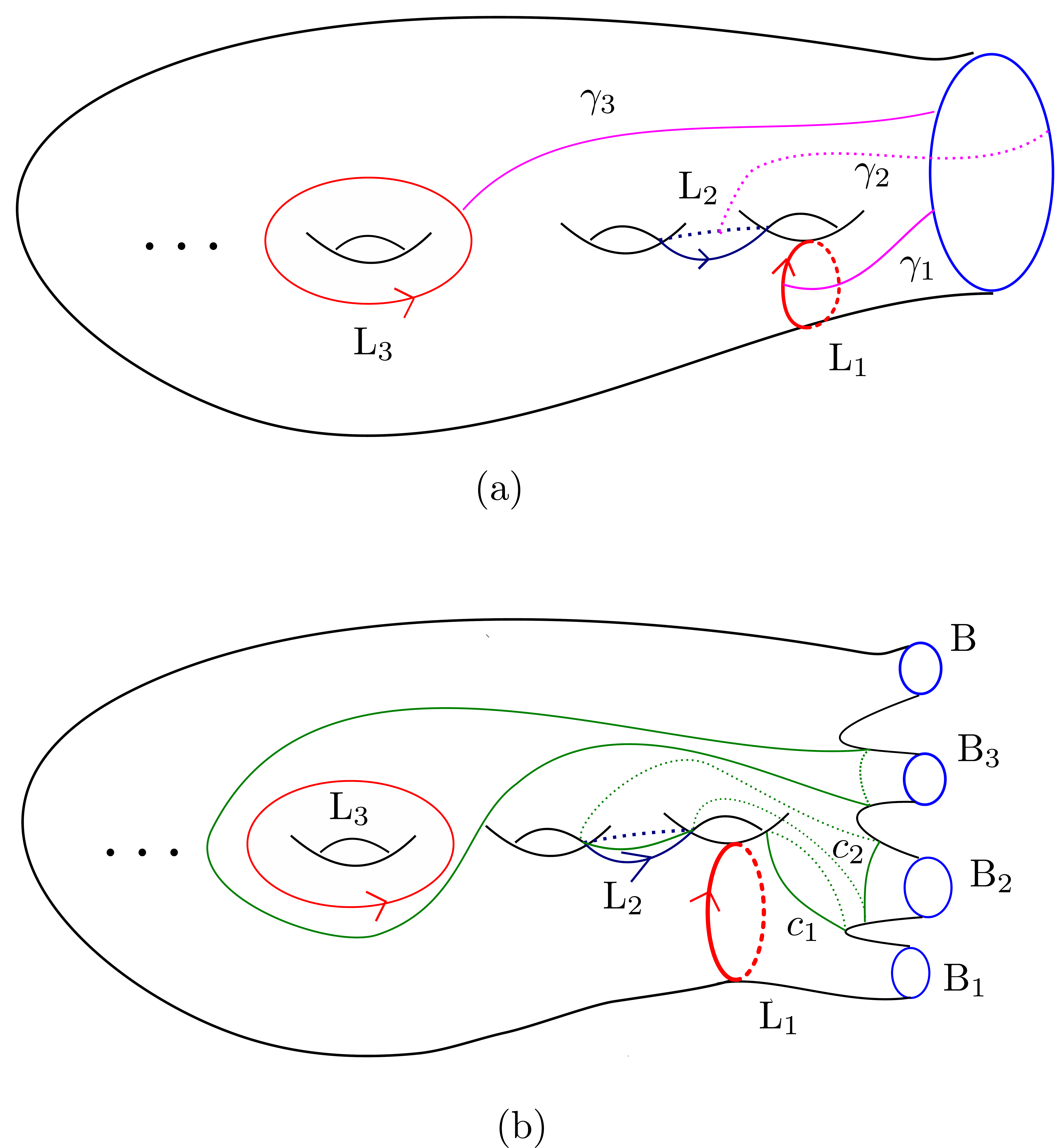}

\caption{(a)Example of an open book where we have mixed type of link components. (b) The resulting open book after we do a Dehn twist along $c_1, c_2$ and $c_3$. }
\label{fig:genusmixedtype}
\end{figure}

\begin{figure}[!htbp]
\centering
\includegraphics[scale=0.1]{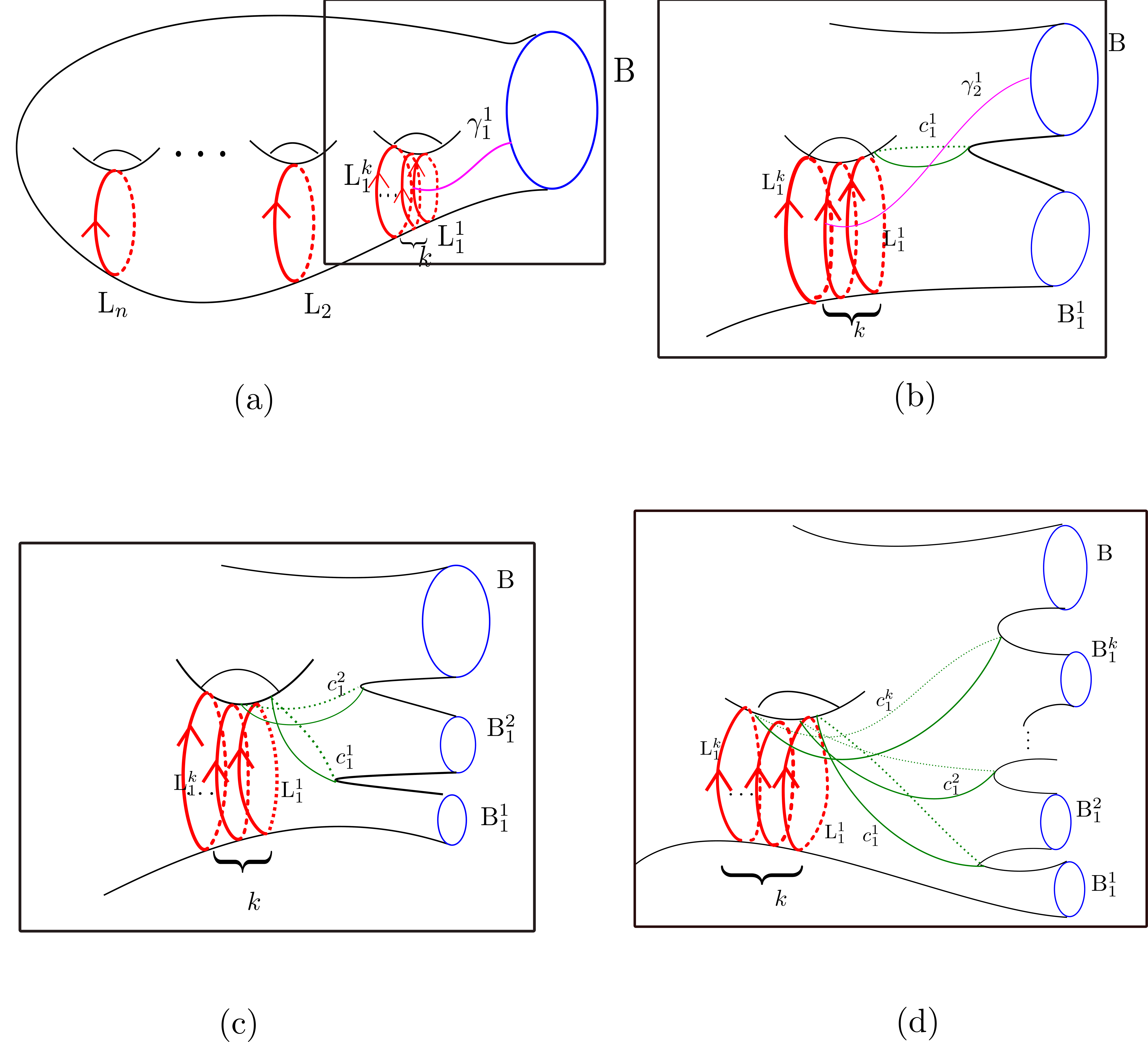}

\caption{(a) $\Li_1$ has $k$ parallel copies. (b) Enlarged view of the region near $\Li_1$. (c) The once stabilized open book along $c_1^1$. Find $\gamma_1^2$ and iterate this process. (d) The final result after $k$ iterated stabilizations of $\B$.}
\label{fig:case2b}
\end{figure}

\subsection{\bf Case 2(b)}
\label{ssec:case2b}
[$(\Sigma_g,\phi)$ has $g>0$ and not all the link components are linearly independent]
 Suppose $\Li_i$ has $k$ parallel copies and we call them $\Li_i^j$ for $j=1,2,\dots k$. This can have the following two sub cases.
 \begin{figure}
 \centering
 \includegraphics[scale=0.10]{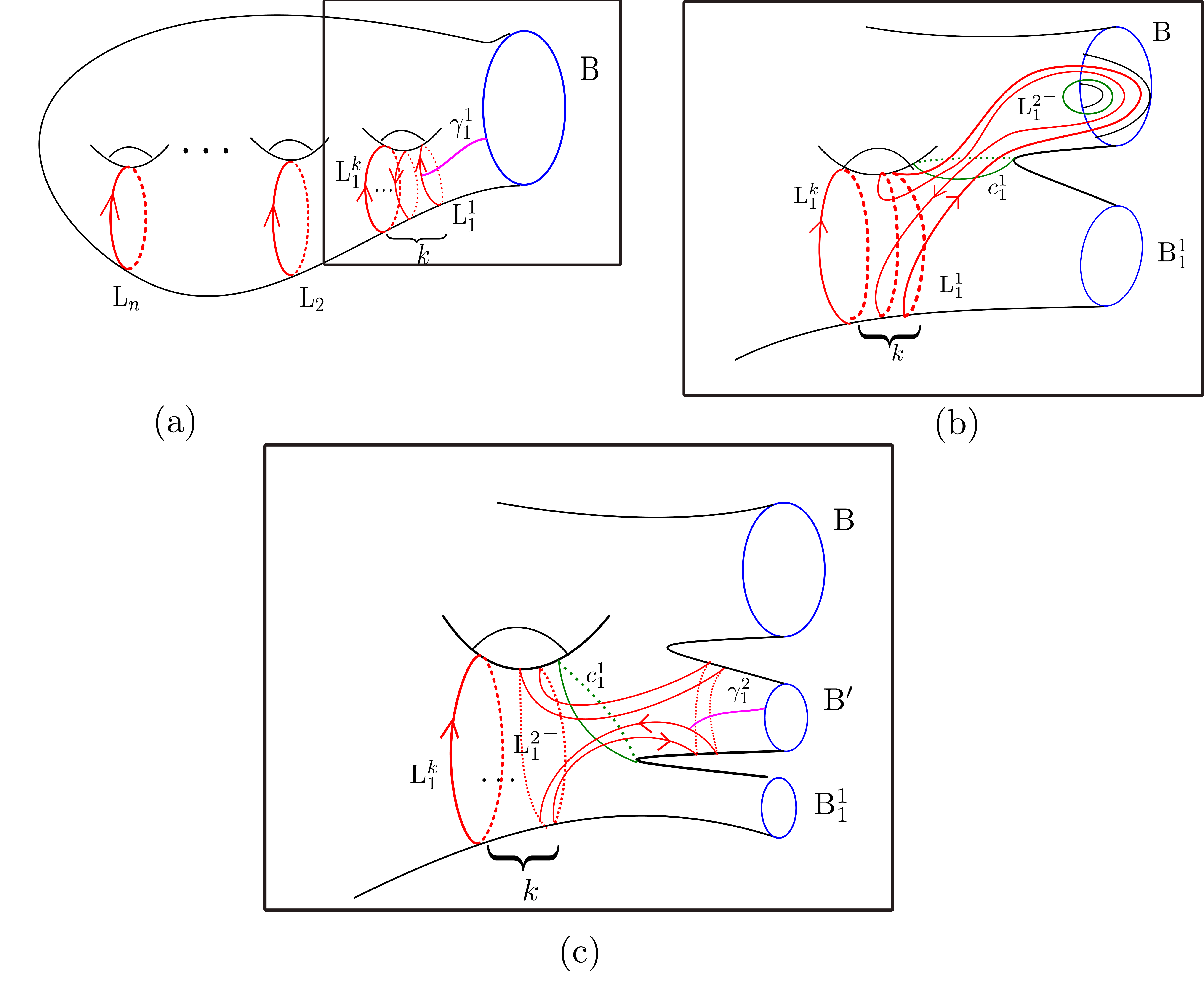}
 \caption{(a) Linearly dependent link components with $\Li_1^2$ having opposite orientation. (b)We stabilize $\B$ along a boundary parallel curve and move the components across the attaching 1-handle. This negatively stabilizes $\L_1^2$. (c) Find $\gamma_1^2$.  }
 \label{fig:genus_distinct_orientation}
 \end{figure}
 \subsubsection{Case 2(b)(i)}[If all the components are oriented similarly]
 
 We consider the local picture \fullref{fig:case2b}(b) where are link component has the same orientation.  Choose the component which is closest to $\B$, call it $\Li_i^1$ and find an arc $\gamma_i^1$ that stays to its right. Notice, this arc does not intersect any of the $\Li_i^j$'s where $j\neq 1$. Do the same procedure as before and find a new boundary component $\B_i^1$. We have the new open book $(\B\sqcup\B_i^1, \Sigma',\phi\circ D_{c_i^1})$. On this new open book, choose an arc $\gamma_i^2$ which can only possibly intersect $c_i^1$ (the closed curve we found using $\gamma_i^1$) and $\Li_i^1$. We iterate this process step by step. This will allow us to find a ordered set of simple closed curves $\{c_i^1, c_i^2,\dots, c_i^k\}$ where $c_i^k$ only possibly intersects $c_i^j$ and $\Li_i^j$ for $j=1,2,\dots,k-1$. Thus if we maintain the order and do the Dehn twist step by step that will not change the other link components we have not dealt with in the previous steps. We finally find a $k$ component link which is the transverse push-off of $\Li_i^1\cup \Li_i^2\cdots \Li_i^{k}$. We can do this process locally for all link components which are not linearly independent. Combining this with all the cases from \fullref{ssec:case2a} gives the desired result for every possible cases.
 
 \subsubsection{Case 2(b)(ii)}[If some of the link components have different orientation] To deal with this case, we follow the same procedure as in \fullref{sssec:differentorientation}. Check \fullref{fig:genus_distinct_orientation}.
 \begin{remark}
 Note that, in \fullref{ssec:case2b} we assumed the components to be meridional. The same idea also works for the other cases i.e if the linearly dependent components bound the genus or go between the genus.
 \end{remark}
 \end{proof}

Now we are ready to define the support genus of a transverse link. 
\begin{definition}
 The support genus $\sg(\T)$ of a transverse link $\T$ in a contact $3$-manifold $(\M,\xi)$ is the minimal genus of a page of the open book decomposition of $\M$ supporting $\xi$ such that $\T$ can be realized as a sub-binding of that open book.
\end{definition}
\begin{figure}[!htbp]
\centering
	\includegraphics[scale=0.15]{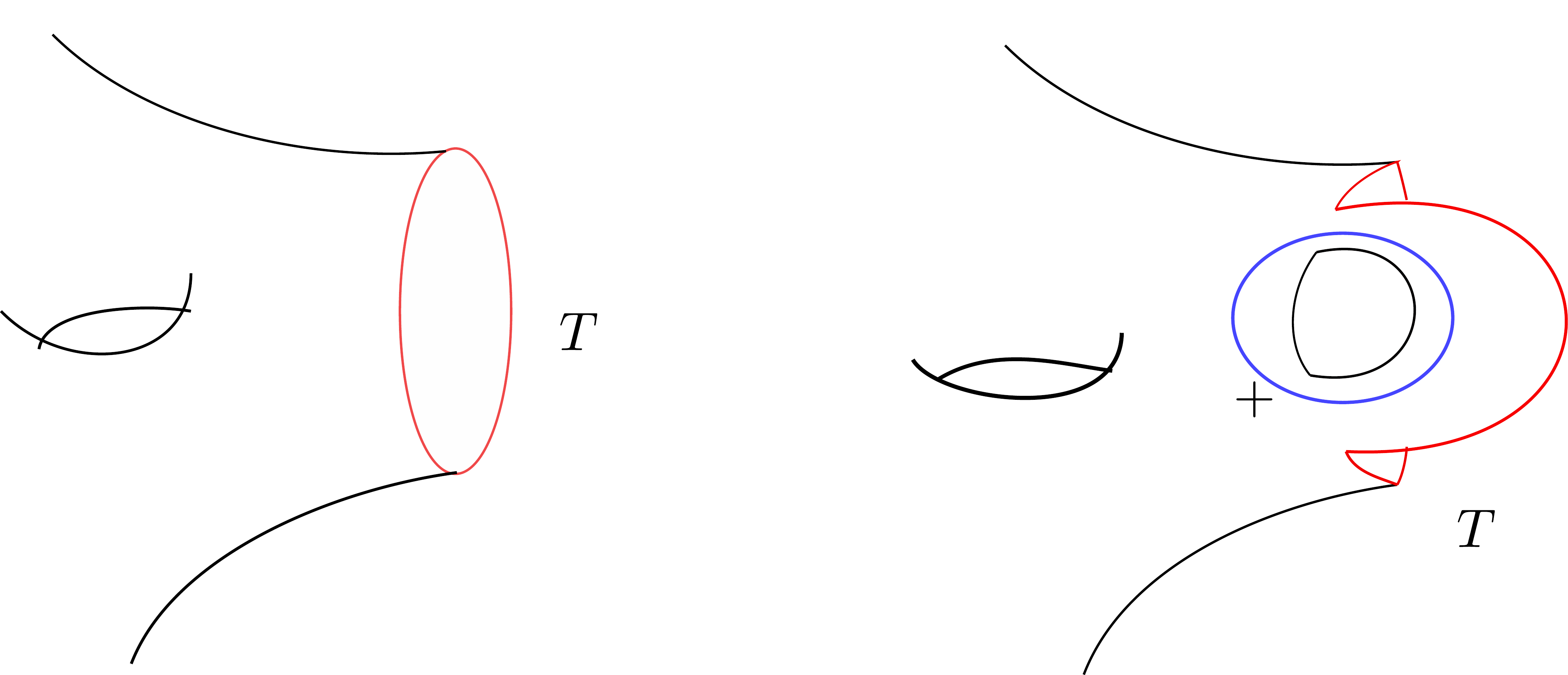}
	\caption{Positive stabilization of the open book along T.}
	
		\label{fig:stabilize}
\end{figure}

\begin{lemma}
	\label{lemma:genus}
Let $T$ be a transverse link in $(M,\xi)$ with $sg(T)=g$. Then its Legendrian approximations with their contact framing $ \leq$ page framing can be put on the page of a supporting open book of $(\M,\xi)$ without altering the genus.
\end{lemma}
\begin{proof}
	Let $(\Sigma_g,\phi)$ be an open book with genus $g$ such that $T$ is a sub-binding of $(\Sigma_g,\phi)$. 
	\subsubsection*{Case 1}
	Suppose that $|\partial\Sigma_g|\geq 2$. In this case, Baldwin and Etnyre showed [Lemma 5.3, \cite{BE}] that a tubular neighborhood of the binding is contactomorphic to the standard neighbourhood of a transverse knot $S_\epsilon$ for some $\epsilon\in (0,1)$. Thus there exists a neighbourhood of the binding whose boundary is foliated by Legendrians which sit of the page of the open book. Notice, in this case the contact framing is same as the page framing. This Legendrian is known as a Legendrian approximation of the binding. Thus we can realize a Legendrian approximation of $T$ on the page of the open book.
		\subsubsection*{Case 2} Suppose the open book has only one component $T$. Then we stabilize the open book positively  along $T$ as shown in \fullref{fig:stabilize}. Notice, this operation neither changes the genus of the open book nor the transverse isotopy class of the binding $T$ as shown in \cite{sv}. Also notice, $T$ is still a sub-binding of the new open book. Now we are in case 1.
	\begin{figure}[!htbp]
\centering
	\includegraphics[scale=0.20]{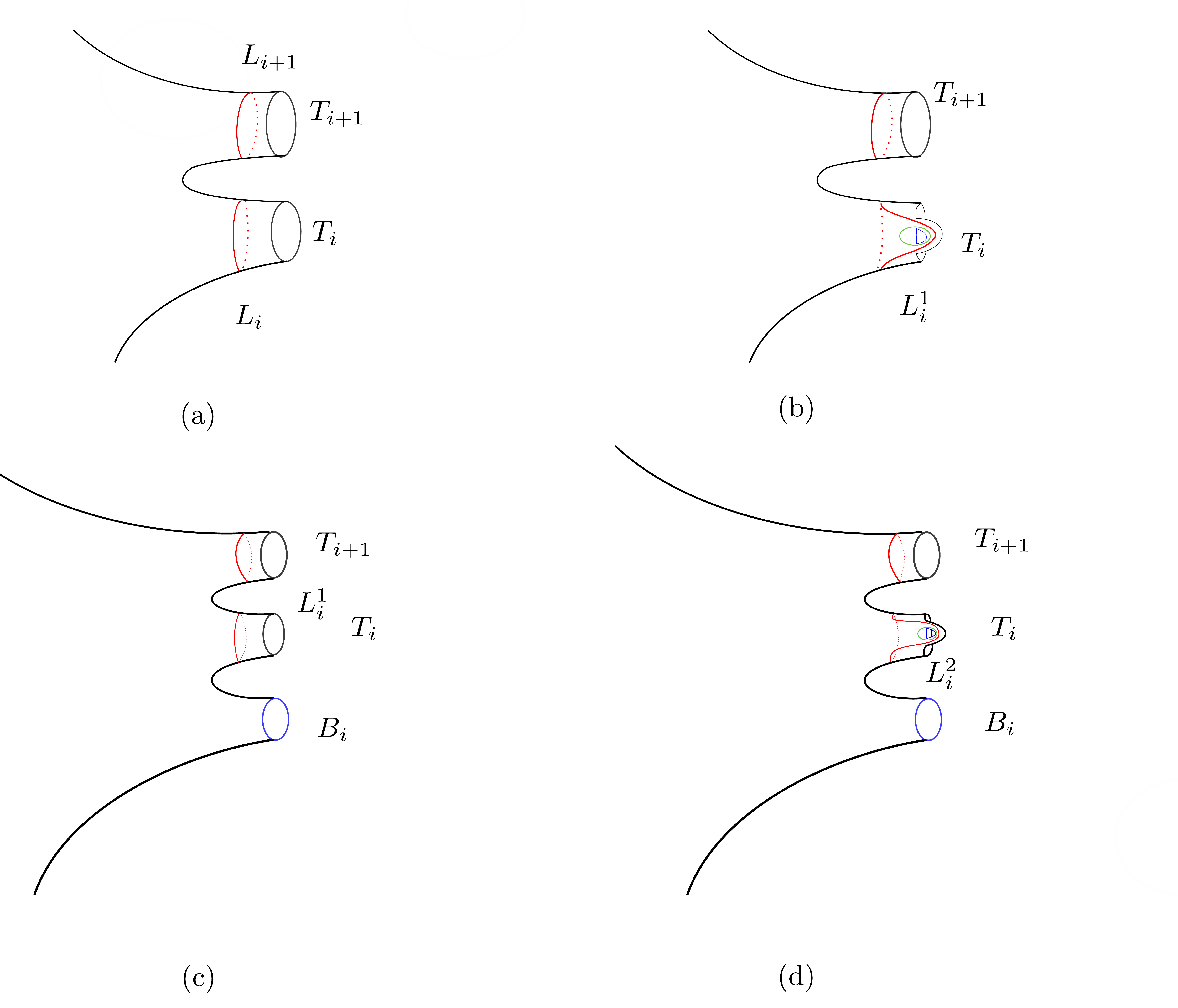}
	\caption{(a) Local picture of an open book with two boundary component.(b)Positively stabilize the open book along $T_i$.(c)Once negatively stabilized $L_i^1$ on the page.(d)Iterate the process and get the Legendrian approximation $L_i^2$. }
		\label{fig:iterated_stabilization}

\end{figure}

	 We can repeat this process along a binding component $T_i$ and push the component over the 1-handle as shown in \fullref{fig:iterated_stabilization}. This gives us the negative stabilization $L_i^1$ of the Legendrian component $L_i$ on the page of the open book \cite{etplanar}, \cite{ona}. So this decreases the contact framing with respect to the page framing by 1.  Repeating this procedure along $T_i$ we can realize all Legendrian approximations of $T$ such that the contact framing is less than or equal to the page framing. This completes the proof. 
	\end{proof}

\begin{remark}
	Notice this result will not generally hold for every Legendrian approximation of $T$. It might not be possible to increase the contact framing of a Legendrian link with respect to the page framing by keeping the genus fixed.
\end{remark}

\begin{theorem}
	\label{thm:invariant}
The support genus is an invariant of a transverse link.
\end{theorem}
\begin{proof}
Suppose  $T_1$ and $T_2$ are transverly isotopic links in $(M,\xi)$ with support genus $g_1$ and $g_2$. So there exists open books $(\Sigma_{g_1}, \phi_1), (\Sigma_{g_2},\phi_2)$ for $T_1$ and $T_2$.
As stated in \cite{geiges}, a transverse link can be considered as a contact submanifold. So, we can extend the transverse isotopy to an ambient contact isotopy by the isotopy extension theorem [Theorem  2.6.12, \cite{geiges}] which takes one open book to the other. Clearly, $g_1=g_2$.
\end{proof}
Our next theorem finds a relationship between the support genus of a transverse link and the support genus of its Legendrian approximation. One might hope that a transverse link will have a support genus greater than its Legendrian approximation, surprisingly that's not the case in this following special case.
\begin{theorem}
\label{thm:transversesg}
Suppose $\T$ is a transverse link in $(\M,\xi)$ and $\Li$ is its Legendrian approximation such it's contact framing $\leq$ page framing. Then $\sg(\T)=\sg(\Li)$.
\end{theorem}
\begin{proof}
We start with a transverse link $\T$ in $(\M,\xi)$ and let $\sg(\T)=g$. So we can realize it as a sub-binding of some open book $(\Sigma,\phi)$ with minimum genus $g$. Now take a Legendrian approximation of each of the components such that the components lie on the page of the open book. This can be done without changing the genus as shown in \fullref{lemma:genus}. This will give us a Legendrian link $L$ sitting on an open book with genus $g$. Thus $\sg(\Li)\leq g$. Now take this Legendrian link and put it on an open book with  genus=$\sg(\Li)$. Now apply the algorithm we used in \fullref{thm:binding} to find a transverse push off $\T'$ of $L$ which is also a sub-binding of the underlying open book. Thus $\sg(\T')\leq\sg(\Li)$. By the well-definedness of transverse push-off $\T'$ must be transversely isotopic to $\T$. As the support genus is an invariant of transverse links by \fullref{thm:invariant}, we must have $\sg(\T')=g$. Thus $\sg(\T)=\sg(\Li)$. As by \fullref{lemma:genus} we can get all Legendrian approximations of $T$ with contact framing $\leq$ page framing without altering the genus, the proof follows.
\end{proof}

The following theorem was proved in \cite{ona}.
\begin{theorem}[\cite{ona}]
	\label{thm:ona}
Suppose $\Li$ is a loose, null-homologous Legendrian knot in $(\M,\xi)$. Then $\sg(\Li)=0$.
\end{theorem}
As a corollary to the above theorem, we show the following:
\begin{corollary}
\label{cor:loosetransverse}
Suppose $\T$ is a loose, null-homologous transverse knot in $(\M,\xi)$. Then $\sg(\T)=0$.
\end{corollary}
\begin{proof}
Suppose $\T$ is a loose, null-homologous transverse knot in $(\M,\xi)$ with $\sg(T)>0$. So we find an open book $(\Sigma_g,\phi)$ such that $T$ can be realized as a sub-binding of the open book with minimum genus g $\neq 0$. By \fullref{lemma:genus}, we can get some Legendrian approximation $L$ of $T$ as a leaf of the characteristic foliation of $\partial S_\epsilon$ such that the contact framing agrees with the page framing. Note that a Legendrian approximation of a loose transverse knot can be loose or non-loose \cite{et}. If $L$ is non-loose, then we can negatively stabilize $L$ enough times so that it becomes loose. Notice that, we can negatively stabilize a Legendrian knot on the page of an open book such that the contact framing agrees with the page framing and keeping the genus fixed as shown in \fullref{lemma:genus}. By abuse of notation, we call the negatively stabilized knot $L$ too. Now applying the algorithm of \fullref{thm:binding}, we can realize the transverse push-off of $L$, say $T'$ as a sub-binding of the planar open book. But as support genus is an invariant of the transverse isotopy class, $sg(T)=\sg(T')$ and we reached a contradiction. Thus $\sg(T)=0$.
\end{proof}

 \bibliographystyle{mwamsalphack}
\bibliography{references1}
\end{document}